\documentclass[letterpaper,11pt,oneside]{amsart}
\usepackage{amsmath}
\usepackage{amssymb}
\usepackage{latexsym}
\usepackage{amsfonts}
\usepackage[latin1]{inputenc}
\usepackage[english]{babel}
\usepackage{multicol}
\usepackage{oldgerm}
\usepackage{enumerate}
\newtheorem{teo}{Theorem}
\newtheorem{lema}{Lemma}

\begin{document}

\title{On the largest prime factor of the $k-$Fibonacci numbers}

\author{Jhon J. Bravo}
\address{Departamento de Matem\'aticas\\ Universidad del Cauca\\ Calle 5 No 4--70\\Popay\'an, Colombia. Current address: Instituto de Matem\'aticas\\ Universidad Nacional Aut\'onoma de M\'exico\\ Circuito Exterior, Ciudad Universitaria\\ C. P. 04510, M\'exico D.F.\\ M\'exico.}
\email{jbravo@unicauca.edu.co}

\author{Florian Luca}
\address{Centro de Ciencias Matem\'aticas \\ Universidad Nacional Aut\'onoma de M\'exico\\ Ap. Postal 61--3 (Xangari), C.P. 58089\\ Morelia, Michoac\'an, M\'exico.}
\email{fluca@matmor.unam.mx}


\date{\today}

\begin{abstract}
Let $P(m)$ denote the largest prime factor of an integer $m\geq 2$, and put $P(0)=P(1)=1$. For an integer $k\geq 2$, let $(F_{n}^{(k)})_{n\geq 2-k}$ be the $k-$generalized Fibonacci sequence which starts with $0,\ldots,0,1$ ($k$ terms) and each term afterwards is the sum of the $k$ preceding terms.  Here, we show that if $n\geq k+2$, then $P(F_n^{(k)})>c\log\log n$, where $c>0$ is an effectively computable constant. Furthermore, we determine  all the $k-$Fibonacci numbers $F_n^{(k)}$ whose largest prime factor is less than or equal to 7.  

\medskip  



\end{abstract}

\maketitle


\section{Introduction}

Let $k\ge 2$ be an integer. We consider a generalization of the Fibonacci sequence called the {\it $k-$generalized Fibonacci
sequence} $F_{n}^{(k)}$ defined as
\begin{equation*}\label{recurrencia}
F_{n}^{(k)}=F_{n-1}^{(k)}+F_{n-2}^{(k)}+\cdots+F_{n-k}^{(k)},
\end{equation*}
with the initial conditions $F^{(k)}_{-(k-2)}=F^{(k)}_{-(k-3)}=\cdots=F^{(k)}_0=0$ and $F^{(k)}_1=1$. We call $F_{n}^{(k)}$ the $n^{th}$ \emph{ $k-$generalized Fibonacci number}, or for simplicity, the $n^{th}$ \emph{$k-$Fibonacci number}. Note that for $k=2$, we obtain the classical Fibonacci sequence
\[
F_0=0, \quad F_1=1 \quad\text{and}\quad F_{n}=F_{n-1}+F_{n-2} \quad\text{for $n\geq 2$}
\]
\[
(F_{n})_{n\ge 0}=\{0,1,1,2,3,5,8,13,21,34,55,89,144,233,377,610,\ldots\}.
\]
If $k=3$, the Tribonacci sequence appears
\[
(T_{n})_{n\ge -1}=\{0,0,1,1,2,4,7,13,24,44,81,149,274,504,927,1705,\ldots\}.
\]
If $k=4$, we get the Tetranacci sequence
\[
(F_{n}^{(4)})_{n\ge -2}=\{0, 0, 0, 1, 1, 2, 4, 8, 15, 29, 56, 108, 208, 401,773,1490,\ldots\}.
\]
For every integer $m$, let $P(m)$ denote the largest prime factor of $m$, with the usual convention that $P(0)=P(\pm 1)= 1$.  The problem of finding lower bounds for the greatest prime factor of terms of linear recurrence sequences has attracted great attention from several number--theorists.  There are many papers in the literature which address interesting results about this problem. In this paper, we follow the same approach for the $k-$generalized Fibonacci sequence; that is, we are interested in finding effective lower bounds for $P (F_n^{(k)})$ in terms of both the parameters $k$ and $n$.  

We prove the following result, which in particular shows that for every $k\geq 2$, $P(F_n^{(k)})\rightarrow \infty$ as $n\rightarrow\infty$. 

\begin{teo} \label{teo1}
The inequality 
\[
 P(F_n^{(k)})>\frac{1}{84}\log\log n
\]
holds for all $n\geq k+2$. 
\end{teo}

Our method is roughly as follows. We use lower bounds for linear forms in logarithms of algebraic numbers to bound $\log n$ in terms of $k$ and $P(F_n^{(k)})$. The result is obtained easily when $k$ is small.  When $k$ is large, we use the fact that the dominant root of the $k-$generalized Fibonacci sequence is exponentially close to $2$, so we can replace this root by $2$ in our calculations with linear forms in logarithm and finish the job.

The same ideas mentioned in the previous paragraph and the LLL algorithm will be used at the end of the paper, in order to find all the $k-$Fibonacci numbers whose greatest prime factor is less than or equal to 7.


\section{Some tools}

We begin by noting that the first $k+1$ non--zero terms in the $k-$generalized Fibonacci sequence are powers of two, namely
\begin{equation*}\label{1-eros}
F_{1}^{(k)}=1,\,\,F_{2}^{(k)}=1,\,\, F_{3}^{(k)}=2,\,\, F_{4}^{(k)}=4,\,\ldots,\, F_{k+1}^{(k)}=2^{k-1},
\end{equation*}
while the next term in the above sequence is $F_{k+2}^{(k)}=2^{k}-1$. As a matter of fact, in \cite{BL2}, we showed that if $n\geq k+2$, then the only $k-$Fibonacci number which is a power of 2 is $F_6=8$.  From now on, we can assume $n\geq k+2$, since otherwise $P(F_n^{(k)})$ is either 1 or 2.

On the other hand, it is known that the characteristic polynomial of the $k-$generalized Fibonacci sequence $(F_n^{(k)})_n$, namely 
\[
\Psi_k(x)=x^k-x^{k-1}-\cdots-x-1,
\]
is irreducible over $\mathbb{Q}[x]$ and has just one root outside the unit circle. Throughout this
paper, $\alpha:=\alpha(k)$ denotes that single root, which is located between $2(1-2^{-k})$ and 2 (see \cite{DAW}).  To simplify notation, in general we omit the dependence  on $k$ of $\alpha$.

The following ``Binet--like formula" for $F_n^{(k)}$ appears in  Dresden \cite{GPD}:
\begin{equation}
\label{eq:binet}
F_n^{(k)}=\sum_{i=1}^{k}\frac{\alpha_i-1}{2+(k+1)(\alpha_i-2)}\alpha_i^{n-1},
\end{equation}
where $\alpha=\alpha_1,\ldots,\alpha_k$ are the roots of $\Psi_k(x)$. It was also proved in \cite{GPD} that the contribution of the roots which are inside the unit circle 
to the formula \eqref{eq:binet} is very small, namely that the approximation 
\begin{equation} 
\label{error}
\left|F_n^{(k)}-\frac{\alpha-1}{2+(k+1)(\alpha-2)}\alpha^{n-1}\right|<\frac{1}{2}  \quad \text{holds~for all}\quad n\geq 2-k.
\end{equation}
We will use the estimate (\ref{error}) later. Furthermore, in \cite{BL} we proved that
\begin{equation}\label{deskfib}
\alpha^{n-2} \leq F_n^{(k)}\leq \alpha^{n-1} \quad \text{for all}\quad n\geq 1.
\end{equation}
 
For $n\geq 3$, let $F_n^{(k)}= p_1^{\beta_1}\cdots p_s^{\beta_s}$ be the prime factorization of the positive integer $F_n^{(k)}$, where $2=p_1 <\cdots <p_r <\cdots$ is the increasing sequence of prime numbers, and the numbers $\beta_i$ for $i=1,\ldots,s$ are nonnegative integers with $\beta_s\geq 1$.

By the right--hand side of inequality \eqref{deskfib}, we have that $p_1^{\beta_1}\cdots p_s^{\beta_s}\leq \alpha^{n-1}<2^{n-1}$. Thus,
\[
\sum_{i=1}^{s}\beta_i\log 2\leq \sum_{i=1}^{s}\beta_i\log p_i<(n-1)\log 2,
\]
giving $\sum_{i=1}^{s}\beta_i <n-1$. In particular
\begin{equation}\label{util1}
\beta_i<n-1 \quad \text{for all} \quad i=1,\ldots,s.
\end{equation}

If $k\geq 3$, then it is a straightforward exercise to check that $1/\log\alpha<2$ by using the fact that $2(1-2^{-k})<\alpha$. If $k=2$, then $\alpha$ is the golden section so $1/\log\alpha=2.078\ldots<2.1$. In any case, the inequality $1/\log\alpha<2.1$ holds for all $k\geq 2$. We record this estimate for future referencing.

To conclude this section, we consider for an integer $r\geq 2$, the function 
\begin{equation}\label{fun-f}
f_r(x)=\frac{x-1}{2+(r+1)(x-2)} \quad \text{for} \quad x>2(1-2^{-r}). 
\end{equation}
We can easily see that 
\begin{equation}
\label{eq:fsprime}
f_r^{\prime}(x)=\frac{1-r}{(2+(r+1)(x-2))^2}\quad \text{for all}\quad x>2(1-2^{-r}),
\end{equation}
and $2+(r+1)(x-2)\geq 1$ for all $x>2(1-2^{-r})$ and $r\geq 3$. We shall use this fact later.


\section{Preliminary estimate} \label{prelim-estimate}

Here, we will use a linear form in logarithms to get an inequality involving $n,k$ and $s$. This puts a bound on $\log n$ in terms of $\log k$ and $\log s$. Because of our assumptions, we have that $F_n^{(k)}$, except in the case $F_6=8$, is not a power of 2. Hence, we can suppose $s\geq 2$ and so we get easily that $n\geq 4$ and $p_s\geq 3$.

By using the prime factorization of $F_n^{(k)}$ and (\ref{error}), we obtain that 
\begin{equation}\label{des-al-beta}
|p_1^{\beta_1}\cdots p_s^{\beta_s}-f_k(\alpha)\alpha^{n-1}|< \frac{1}{2}.
\end{equation}
Dividing both sides of the above inequality by $f_{k}(\alpha)\alpha^{n-1}$, which is positive because $\alpha>1$ and 
$2^{k}>k+1$, so $2>(k+1)(2-(2-2^{-k+1}))>(k+1)(2-\alpha)$, we obtain the inequality
\begin{equation} \label{deslambda}
\left|p_1^{\beta_1}\cdots p_s^{\beta_s}\cdot\alpha^{-(n-1)}\cdot(f_{k}(\alpha))^{-1}-1\right|<\frac{2}{\alpha^{n-1}},
\end{equation}
where we used the facts $2+(k+1)(\alpha-2)<2$ and $1/(\alpha-1)<2$, which are easily seen. 

We shall need a result of E.M. Matveev \cite{Matveev} about linear forms in logarithms. But first, some notation. For an algebraic number $\eta$ we write $h(\eta)$ for its logarithmic height, whose formula is
\[
h(\eta):=\frac{1}{d}\left(\log a_0+\sum_{i=1}^{d}\log\left(\max\{|\eta^{(i)}|,1\}\right)\right),
\]
with $d$ being the degree of $\eta$ over $\mathbb{Q}$ and
\begin{equation}
\label{eq:f}
f(X):=a_0\prod_{i=1}^{d}(X-\eta^{(i)}) \in \mathbb{Z}[X]
\end{equation}
being the minimal primitive polynomial over the integers having positive leading coefficient $a_0$ and $\eta$ as a root.

With this notation, Matveev (see \cite{Matveev} or Theorem 9.4 in \cite{Bug}) proved the following deep result:

\begin{lema}\label{teoMatveev}
Let $\mathbb{K}$ be a number field of degree $D$ over $\mathbb{Q}$, $\gamma_1, \ldots, \gamma_t$ be positive reals of $\mathbb{K}$, and $b_1,\ldots,b_t$ rational integers. Put 
\[
B\geq \max\{|b_1|,\ldots,|b_t|\},
\]
and 
\[
\Lambda:=\gamma_1^{b_1}\cdots\gamma_t^{b_t}-1.
\]
Let $A_1,\ldots,A_t$ be real numbers such that 
\[
A_i\geq \max\{Dh(\gamma_i),|\log \gamma_i|, 0.16\}, \quad \text{$i=1,\ldots,t$}.
\]
Then, assuming that $\Lambda\neq 0$, we have
\[
|\Lambda|>\exp\left(-1.4\times 30^{t+3}\times t^{4.5}\times D^2(1+\log D)(1+\log B)A_1\cdots A_t\right).
\]
\end{lema}

In order to apply Lemma \ref{teoMatveev}, we take $t:=s+2$ and
\[
\gamma_i:=p_i \quad\text{for all} \quad i=1,\ldots,s,\quad \gamma_{s+1}:=\alpha, \quad \gamma_{s+2}:=f_k(\alpha).
\]
We also take the exponents $b_i:=\beta_i$ for all $i=1,\ldots,s$, $b_{s+1}:=-(n-1)$ and $b_{s+2}:=-1$. Hence, 
\begin{equation*}\label{deflambda}
\Lambda:=\prod_{i=1}^{s+2}\gamma_i^{b_i}-1.
\end{equation*}
Observe that the absolute value of $\Lambda$ appears in the left--hand side of inequality (\ref{deslambda}). The algebraic number field containing $\gamma_i's$ is $\mathbb{K}:=\mathbb{Q}(\alpha)$. As $\alpha$ is of degree $k$ over $\mathbb{Q}$, it follows that $D:=[\mathbb{K}:\mathbb{Q}]=k$. To see that $\Lambda\neq 0$, observe that imposing that $\Lambda=0$ yields
\[
p_1^{\beta_1}\cdots p_s^{\beta_s}=\frac{\alpha-1}{2+(k+1)(\alpha-2)}\alpha^{n-1}.
\]
Conjugating the above relation by some automorphism of the Galois group of the splitting field of $\Psi_k(x)$ over $\mathbb{Q}$ and then taking absolute values, we get that for any $i > 1$, we have
\[
p_1^{\beta_1}\cdots p_s^{\beta_s}=\left|\frac{\alpha_i-1}{2+(k+1)(\alpha_i-2)} \alpha_i^{n-1}\right|.
\]
But the above relation is not possible since its left--hand side is greater than or equal to $3$, while its right--hand side is smaller than 
$2/(k-1)\leq 2$ because $|\alpha_i|<1$ and
\begin{equation}\label{eq:useful2}
|2+(k+1)(\alpha_i-2)|\geq (k+1)|\alpha_i-2|-2>k-1.
\end{equation}
Thus, $\Lambda\neq 0$.

Since $h(\gamma_i)=\log p_i\leq \log p_s$ for all $i=1,\ldots,s$, it follows that we can take $A_i:=k\log p_s$ for all $i=1,\ldots,s$. Furthermore, since $h(\gamma_{s+1})=(\log\alpha)/k<(\log 2)/k=(0.693147\ldots)/k$, it follows that we can take $A_{s+1}:=0.7$.   

We now need to estimate $h(\gamma_{s+2})$. First, observe that
\begin{equation}\label{desh3} 
h(\gamma_{s+2})=h(f_k(\alpha))=h\left(\frac{\alpha-1}{2+(k+1)(\alpha-2)}\right).
\end{equation}

Put
\[
g_{k}(x)=\prod_{i=1}^{k}\left(x-\frac{\alpha_i-1}{2+(k+1)(\alpha_i-2)}\right)\in {\mathbb Q}[x].
\]
Then the leading coefficient $a_0$ of the minimal polynomial of $(\alpha-1)/(2+(k+1)(\alpha-2))$ over the integers (see definition \eqref{eq:f}) divides $\prod_{i=1}^{k}(2+(k+1)(\alpha_i-2))$. But,
\[
\left|\prod_{i=1}^{k}(2+(k+1)(\alpha_i-2))\right|=(k+1)^{k}\left|\prod_{i=1}^{k}\left(2-\frac{2}{k+1}-\alpha_i\right)\right|=(k+1)^{k}\left|\Psi_{k}\left(2-\frac{2}{k+1}\right)\right|.
\]
Since 
$$
|\Psi_{k}(y)|<\max\{y^{k},1+y+\cdots+y^{k-1}\}<2^{k}\quad {\text{\rm for~all}}\quad 0<y<2,
$$
it follows that 
$$
a_0\leq (k+1)^{k} \left|\Psi_{k}\left(2-\frac{2}{k+1}\right)\right|< 2^{k}\,(k+1)^{k}.
$$
Hence, 
\begin{align}
h\left(\frac{\alpha-1}{2+(k+1)(\alpha-2)}\right)=&\,
\frac{1}{k}\left(\log a_0+\sum_{i=1}^{k}\log\max\left\{\left|\frac{\alpha_i-1}{2+(k+1)(\alpha_i-2)}\right|,1\right\}\right)\notag\\
<&\,\frac{1}{k}\,(k\log 2+k\log(k+1))\notag\\
=&\,\log(k+1)+\log 2<3\log k.   \label{halpha31}
\end{align}

In the above inequalities, we used the facts $\log(k+1)+\log 2<3\log k$ for all $k\geq 2$ and
\[
\left|\frac{\alpha_i-1}{2+(k+1)(\alpha_i-2)}\right|<1 \quad\text{for all $1\leq i \leq k$},
\]
which holds because for $k=2$, $\alpha$ is the golden section and so $\alpha_2=(1-\sqrt{5})/2$, thus
\[
\left|\frac{\alpha_2-1}{2+3(\alpha_2-2)}\right|=0.276\ldots<1 \quad\text{and} \quad \left|\frac{\alpha-1}{2+3(\alpha-2)}\right|=0.723\ldots<1, 
\]
while for $k\geq 3$, $|2+(k+1)(\alpha_i-2)|>k-1\geq 2$ for all $i>1$ (see \eqref{eq:useful2}), and $2+(k+1)(\alpha-2)>1$, which is a straightforward exercise to check using the fact that $2(1-2^{-k})<\alpha<2$.

So, combining (\ref{desh3}) and (\ref{halpha31}) we obtain that $h(\gamma_{s+2})< 3\log k$, therefore we can take $A_{s+2}:=3k\log k$. By recalling that $b_i=\beta_i<n-1$ for all $i=1,\ldots,s$ (see \eqref{util1}), we can take $B:=n-1$. Applying Lemma \ref{teoMatveev} to get a lower bound for $|\Lambda|$ and comparing this with inequality (\ref{deslambda}), we get
\begin{equation} \label{aplic1}
\exp(-C(k,s)\,(1+\log (n-1))\,(0.7)\,(3k\log k)\prod_{i=1}^{s}k\log p_s)<\frac{2}{\alpha^{n-1}},
\end{equation}
where
\begin{align*}
C(k,s) :=& \,  1.4\times 30^{s+5}\times (s+2)^{4.5}\times k^2 \times(1+\log k)\\
 <&\,  7.7\times 10^{8}\, 30^s\, s^{4.5}\,k^2\,(1+\log k),
\end{align*} 
where we used the fact that $s+2\leq 2s,$ which holds for all $s\geq 2$.
 
Taking logarithms on both sides of \eqref{aplic1}, we see that
\[
(n-1)\log\alpha-\log 2 < 7.7\times 10^{8}\, 30^s\, s^{4.5}\,k^2\,(1+\log k)\,(1+\log(n-1))\,(0.7)\,(3k\log k)\,(k^s(\log p_s)^s),
\]
which leads to
\[
n-1<2.1\times 10^{10}\,30^s\, s^{4.5}\,k^{s+3}\,\log^2 k\,\log(n-1)\,(\log p_s)^s,
\]
where we used the facts $1+\log k\leq 3\log k$ for all $k\geq 2$, $1+\log(n-1)\leq 2\log(n-1)$ for all $n\geq 4$ and $1/\log\alpha<2.1$ for all $k\geq 2$.

Now, by recalling that $k<n-1$ because of $n\geq k+2$, and using the fact that the inequality $p_m<m^2$ holds for all $m\geq 2$, we obtain:\footnote{Actually, Corollary from Theorem 3 on p. 69 of \cite{BS} states that $p_m<m(\log m+\log\log m)$ holds for all $m\geq 6$.}
\begin{equation}\label{deslog}
\frac{n-1}{\log^3(n-1)}<2.1\times 10^{10}\,(60\log s)^s\,s^{4.5}\,k^{s+3}.
\end{equation}

One can check that for all $A>10$, the inequality $x/\log^3 x <A$ implies $x<64 A\log^3 A$. To see why, assume that $x\ge 64 A\log^3 A$. Since the function $x\mapsto x/\log^3 x$ is increasing for $x>e^3$, and $64A \log^3 A >e^3$ when $A>4$, we get that
\[
A>\frac{x}{\log^3 x}\geq \frac{64 A\log^3 A}{\log^3(64A\log^3 A)}.
\]

Keeping just the left and the right hand sides of the above inequality and omitting the middle term, we get an  inequality equivalent to 
\[
64A\log^3 A>A^4,
\]
or $4\log A>A$, or $A^4>e^A$, which is false for $A>10$.  Applying this with $A:=2.1\times 10^{10}\,(60\log s)^s\,s^{4.5}\,k^{s+3}$, we get that inequality (\ref{deslog}) yields
\[
n-1<\, 1.35\times 10^{12}\,(60\log s)^s\,s^{4.5}\,k^{s+3}\,(\log(2.1\times 10^{10}\,(60\log s)^s\,s^{4.5}\,k^{s+3}))^3.
\]
On the right--most logarithm, we have that
\begin{align*}
\log(2.1\times 10^{10}\,(60\log s)^s\,s^{4.5}\,k^{s+3})<&\,23.8+s\log(60\log s)+4.5\log s+(s+3)\log k\\
<&\, s\,(9s+k),
\end{align*}
where we used the inequalities $23.8/s+\log(60\log s)+(4.5/s)\log s<9s$ and $(1+3/s)\log k\leq (5/2)\log k<k$ which hold for all $s\geq 2$ and $k\geq 2$.  Hence, we get that 
\[
n-1<\,1.35\times 10^{12}\,(60\log s)^s\,s^{7.5}\,k^{s+3}\,(9s+k)^3,
\]
so
\[
n<\,1.36\times 10^{12}\,\,(60\log s)^s\,s^{7.5}\,k^{s+3}\,(9s+k)^3,
\]
giving
\begin{align*}
\log n<&\,\log(1.36\times 10^{12})+s\log(60\log s)+7.5\log s+(s+3)\log k+3\log(9s+k)\\
<&\,28+s\log(60\log s)+7.5\log s+(s+3)\log k+3\log(9s+k)\\
<&\,30s\log s+3s\log k+3\log(9s+k).
\end{align*}

In the last chain of inequalities, we have used that inequalities $s+3< 3s$ and $28+s\log(60\log s)+7.5\log s<30s\log s$  hold for all $s\geq 2$. We record what we have just proved as a lemma.

\begin{lema}\label{cota_nl}
If $n\geq k+2$ and $F_n^{(k)}=p_1^{\beta_1}\cdots p_s^{\beta_s}$ is the prime factorization of $F_n^{(k)}$ with $\beta_s\geq 1$, then the inequality  
\[
\log n<30s\log s+3s\log k+3\log(9s+k)
\]
holds.
\end{lema}


\section{Proof of Theorem \ref{teo1}}

First of all, observe that if $k\leq s$, then it follows from Lemma \ref{cota_nl}, that
 \[
\log n<33s\log s+3\log(10s)<40s\log s, 
 \]
which holds for all $s\geq 2$. So, in this case we get that $p_s>(1/40)\log n$, where we have used the well--known fact that $p_s>s\log s$ holds for all $s\geq 1$ (see, for example, \cite[p. 69]{BS} or \cite{BR}).


We therefore assume for the remainder of this section that $s<k$. Then, the conclusion of Lemma \ref{cota_nl} can be written as
\begin{equation}\label{cota_nl2}
\log n< 33s\log k+3\log(10k) < 40s\log k,
\end{equation}
where we used that inequalities $3\log(10k)<13\log k$ and $33s+13<40s$ hold for all $k\geq 2$ and $s\geq 2$.

We now proceed with the proof by distinguishing two cases.


\subsection{The case $2^{k/2}\leq n$}

Here, we have the following chain of inequalities
\[
k\leq \frac{2}{\log 2}\log n<116s\log k,
\]
where the last inequality follows directly from \eqref{cota_nl2}.  So
\begin{equation}\label{truco}
\frac{k}{\log k}<116s.
\end{equation}

It is well--known and easy to prove that if $A\geq 3$ and $x/\log x<A$, then $x<2A\log A$ (see, for example, \cite[p. 7]{BL2}). Thus, taking $A:=116s$, inequality \eqref{truco} gives us
\begin{align}
k<&\,2(116s)\log(116s)<232s(\log 116+\log s) \notag \\
<&\,232s (4.8+\log s) <\,1856 s\log s, \label{aplicutil1}
\end{align}
where we used the inequality $4.8+\log s<8\log s$ valid for all $s\geq 2$. Therefore
\begin{align}
\log k<&\,\log(1856)+\log s+\log\log s<\,7.6+\log s+\log\log s \notag\\
<&\,12\log s, \label{aplicutil2}
\end{align}
which holds for all $s\geq 2$.  Finally, we use Lemma \ref{cota_nl} once again and inequalities \eqref{aplicutil1} and \eqref{aplicutil2}, to conclude that
\begin{align*}
\log n<&\,66s\log s+3\log(9s+1856s\log s)\\
<&\,84s\log s,
\end{align*}
which also holds for all $s\geq 2$. Consequently, $p_s>s\log s>(1/84)\log n$. 


\subsection{The case $n<2^{k/2}$}  \label{n<2^k/2}

We treat this case as follows.  Let $\lambda>0$ be such that $\alpha+\lambda=2$. Since $\alpha$ is located between $2(1-2^{-k})$ and 2, we get that $\lambda<2-2(1-2^{-k})=1/2^{k-1}$, i.e., $\lambda\in (0,1/2^{k-1})$.  Besides,
\[
\alpha^{n-1}=(2-\lambda)^{n-1}=2^{n-1}\left(1-\frac{\lambda}{2}\right)^{n-1}=\,2^{n-1}e^{(n-1)\log(1-\lambda/2)}\geq 2^{n-1}e^{-\lambda(n-1)},
\]
where we used the fact that $\log(1-x)\geq -2x$ for all $x<1/2$. But we also have that $e^{-x}\geq 1-x$ for all $x\in\mathbb{R}$, so, $\alpha^{n-1}\geq 2^{n-1}(1-\lambda(n-1))$.

Moreover, $\lambda(n-1)<(n-1)/2^{k-1}<2^{k/2}/2^{k-1}=2/2^{k/2}$. Hence, 
\[
\alpha^{n-1}>2^{n-1}(1-2/2^{k/2}).
\]
It then follows that the following inequalities hold
\[
2^{n-1}-\frac{2^n}{2^{k/2}}<\alpha^{n-1}<2^{n-1}+\frac{2^n}{2^{k/2}},
\]
or
\begin{equation}\label{cotaeta}
\left|\alpha^{n-1}-2^{n-1}\right|<\frac{2^n}{2^{k/2}}.
\end{equation}
We now consider the function $f_k(x)$ given by (\ref{fun-f}).  Using the Mean--Value Theorem, we get that there exists some $\theta\in (\alpha,2)$ such that $f_k(\alpha)=f_k(2)+(\alpha-2)f_k^{\prime}(\theta)$.

Observe that when $k\geq 3$, we obtain 
$$
|f_{k}^{\prime}(\theta)|=(k-1)/(2+(k+1)(\theta-2))^2<k
$$ 
(see the inequality \eqref{eq:fsprime} and the comment following it), and when $k=2$, we have that $\alpha$ is the golden section and therefore $|f_{2}^{\prime}(\theta)|=1/(3\theta-4)^2<25/16$, since $\theta>\alpha>8/5$.  In any case, we obtain $|f_{k}^{\prime}(\theta)|<k$. Hence,
\begin{equation}\label{cotadelta}
|f_k(\alpha)-f_k(2)|=|\alpha-2||f_k^{\prime}(\theta)|=\lambda\,|f_k^{\prime}(\theta)|<\frac{2k}{2^k},
\end{equation}

Writing 
$$
\alpha^{n-1}=2^{n-1}+\delta\quad {\text{\rm and}}\quad f_k(\alpha)=f_k(2)+\eta,
$$ 
then inequalities (\ref{cotaeta}) and (\ref{cotadelta}) yield 
\begin{equation}
\label{eq:deltaeta}
|\delta|<\frac{2^n}{2^{k/2}}\quad {\text{\rm and}}\quad |\eta|<\frac{2k}{2^{k}}.
\end{equation}
Besides, since $f_k(2)=1/2$ for all $k\geq 2$, we have
\begin{equation}
\label{eq:ef}
f_k(\alpha)\,\alpha^{n-1}=(f_k(2)+\eta)(2^{n-1}+\delta)=2^{n-2}+\frac{\delta}{2}+2^{n-1}\eta+\eta\delta.
\end{equation}
So, from (\ref{des-al-beta}) and the inequalities \eqref{eq:deltaeta} and \eqref{eq:ef} above, we get
\[
|p_1^{\beta_1}\cdots p_s^{\beta_s}-2^{n-2}|=\left|\left(F_n^{(k)}-f_k(\alpha)\alpha^{n-1}\right)+\frac{\delta}{2}+2^{n-1}\eta+\eta\delta\right|<\frac{1}{2}+\frac{2^{n-1}}{2^{k/2}}+\frac{2^{n}k}{2^k}+\frac{2^{n+1}k}{2^{3k/2}}.	
\]

Factoring $2^{n-2}$ in the right--hand side of the above inequality and taking into account that $1/2^{n-1}<1/2^{k/2}$ (because $n\geq k+2$), $4k/2^k<5/2^{k/2}$ and $8k/2^{3k/2}\leq 4/2^{k/2}$, which are both valid for $k\geq 2$, we conclude that  
\begin{equation}\label{deslambda1}
\left|p_1^{\beta_1-n+2}\cdot p_2^{\beta_2}\cdots p_s^{\beta_s}-1\right|<\frac{12}{2^{k/2}}.
\end{equation}

We now set 
\begin{equation*} \label{deflambda1}
\Lambda_1:=p_1^{\beta_1-n+2}\cdot p_2^{\beta_2}\cdots p_s^{\beta_s}-1.
\end{equation*}
Observe that $\Lambda_1\neq 0$. Indeed, for if $\Lambda_1=0$, then $F_n^{(k)}=2^{n-2}$, which is impossible because $F_n^{(k)}<2^{n-2}$ for all $n\geq k+2$ (see \cite[Lemma 1]{BL2}).  We lower bound the left--hand side of inequality \eqref{deslambda1} using again Matveev's theorem. We now take $t:=s$ and $\gamma_i:=p_i$ for all $i=1,\ldots,s$. We also take the exponents $b_1:=\beta_1-n+2$ and $b_i:=\beta_i$ for all $i=2,\ldots,s$.  In this application of Matveev's result, we take $D:=1$ and $A_i:=\log p_s$, since $h(\gamma_i)=\log p_i\leq \log p_s$ for all $i=1,\ldots s$. By recalling that $\beta_i<n-1$ for all $i=1,\ldots,s$, we can also take $B:=n$.  We thus get that 
\[
\exp(-C(s)\,(1+\log n)\,\prod_{i=1}^{s}\log p_s)<\frac{12}{2^{k/2}},
\]
where $C(s):=1.4\times 30^{s+3}\times s^{4.5}<37800\times 30^{s}\,s^{4.5}$.

Taking logarithms in the above inequality, we have that
\[
\frac{k}{2}\log 2-\log 12<37800\times 30^{s}\,s^{4.5}\,(1+\log n)(\log p_s)^s.
\]
This leads to
\begin{align*}
k<&\,\frac{2\times 37800\times 30^{s}\,s^{4.5}\,(1+\log n)(\log p_s)^s}{\log 2}+\frac{2\log 12}{\log 2}\\
<& \,218135.5\times 30^{s}\,s^{4.5}\,\log n\,(\log p_s)^s+7.17\\
<&\, 218136\times 30^{s}\,s^{4.5}\,\log n\,(\log p_s)^s.
\end{align*}

In the above, we used the inequality $1+\log n<2\log n$ valid for all $n\geq 3$. But, recall that $p_s<s^2$ for all $s\geq 2$ and $\log n<40s\log k$ from \eqref{cota_nl2}. Thus,
\[
k<8.8\times 10^6\, (60\log s)^s\,s^{5.5}\,\log k, 
\]
which implies
\[
\frac{k}{\log k}<8.8\times 10^{6} \, (60\log s)^s\,s^{5.5}.
\]
Hence,
\begin{align*}
k<&\,2(8.8\times 10^{6} \, (60\log s)^s\,s^{5.5})\log(8.8\times 10^{6} \, (60\log s)^s\,s^{5.5})\\
<&\,(1.76\times 10^{7} \, (60\log s)^s\,s^{5.5}) (16+s\log(60\log s)+5.5\log s)\\
<&\,3.52\times 10^8 (60\log s)^s\,s^{6.5}\,\log s,
\end{align*}
where we used the fact that the inequality 
$$
16+s\log(60\log s)+5.5\log s<20s\log s
$$
holds for all $s\geq 2$. Consequently,
\begin{align}
\log k<&\,\log(3.52\times 10^8)+s\log(60\log s)+6.5\log s+\log\log s \notag\\
<&\, 20+s\log(60\log s)+6.5\log s+\log\log s  \notag\\
<&\,23\,s\log s.  \label{cotafinal}
\end{align}

To finish the proof, recall we are treating the case when $n<2^{k/2}$, therefore $\log n<(k/2)\log 2<k$. This and inequality \eqref{cotafinal} tell us that $\log\log n<\log k<23\,s\log s$, hence $p_s>s\log s>(1/23) \log\log n$.


\section{Numerical theorem}

In this section, we are interested in finding those $k-$Fibonacci numbers whose largest prime factor is less than or equal to 7, i.e., we determine all the solutions of the Diophantine equation 
\begin{equation}  \label{eq2}
F_n^{(k)}=2^a\cdot 3^b\cdot 5^c \cdot 7^d,
\end{equation}
in nonnegative integers $n, k, a, b, c, d$  with $k\geq 2$.

First of all, note that it suffices to consider the case when $n\geq k+2$,  otherwise \eqref{eq2} holds trivially, since the first $k+1$ nonzero terms in the $k-$generalized Fibonacci sequence are powers of two. 

We have the following result.

\begin{teo}  \label{teo2}
The solutions of the Diophantine equation \eqref{eq2} in nonnegative integers $n, k, a, b, c, d$  with $k\geq 2$ and $n\geq k+2$, are:

\vspace{0.3cm}

\begin{center}
\begin{tabular}{ | l | l | l | }
\hline
$F_4^{(2)}=3$   &  $ F_9^{(3)}=81=3^4$ &  $ F_8^{(6)}=63=3^2\cdot 7$ \\ \hline

$F_5^{(2)}=5$   &  $ F_{12}^{(3)}=504=2^3\cdot 3^2\cdot 7$ &  $ F_9^{(6)}=125=5^3$ \\ \hline

$F_6^{(2)}=8=2^3$  &  $ F_{15}^{(3)}=3136=2^6\cdot 7^2$ &  $ F_{14}^{(6)}=3840=2^8\cdot 3\cdot 5$ \\ \hline

$F_8^{(2)}=21=3\cdot 7$  &  $ F_{6}^{(4)}=15=3\cdot 5$ &  $ F_{11}^{(7)}=504=2^3\cdot 3^2 \cdot 7$  \\ \hline

$F_{12}^{(2)}=144=2^4\cdot 3^2 $  &  $ F_8^{(4)}=56=2^3\cdot 7$ &  $ F_{13}^{(7)}=2000=2^4\cdot 5^3$  \\ \hline

 $F_{5}^{(3)}=7 $  &  $ F_9^{(4)}=108=2^2\cdot 3^3$ &  $ F_{16}^{(8)}=16128=2^8\cdot 3^2\cdot 7$  \\ \hline

 $F_{7}^{(3)}=24=2^3\cdot 3 $  &  $ F_9^{(5)}=120=2^3\cdot 3\cdot 5$ &    \\ \hline
\end{tabular}
\end{center}
\end{teo}
\vspace{0.3cm}

Proceeding in a way similar to the above sections, we obtain the following estimates. 

\begin{lema} \label{lema-s=4}     If $P(F_n^{(k)})\leq 7$, then: 
\begin{enumerate}
\item[$(i)$] The inequality
\[
n-1<7.73\times 10^{20}\,k^7\,\log^3k,
\]
holds for all $k\geq 2$ and $n\geq 4$. 

 \item[$(ii)$] Furthermore, if $k>900$, then we get the following absolute upper bounds for $k$ and $n$:
\[
k<1.289\times 10^{17}  \quad \text{and}    \quad n<2.795\times 10^{145}.
\]  
\end{enumerate}
\end{lema}

\begin{proof}
The first part is deduced by using the same arguments as in Section \ref{prelim-estimate}. Indeed, taking $s=4$, inequality \eqref{deslambda} is transformed in
\begin{equation} \label{deslambda for s=4}
\left|2^{a}\cdot 3^{b}\cdot 5^{c}\cdot 7^{d}\cdot\alpha^{-(n-1)}\cdot(f_{k}(\alpha))^{-1}-1\right|<\frac{2}{\alpha^{n-1}}.
\end{equation}

Now, our desired result is obtained after applying linear forms in logarithms to lower bound the left--hand side of inequality \eqref{deslambda for s=4}.

For the second part, assume that $k> 900$. For such $k$ we have the following chain of inequalities
\[
n-1<7.73\times 10^{20}\,k^7\,\log^3k<2^{k/2}.
\]
In particular, $n-1<2^{k/2}$. Hence, we can use the same ideas from Subsection \ref{n<2^k/2} to conclude that
\begin{equation}\label{deslambda1 for s=4}
\left|2^{a-n+2}\cdot 3^{b}\cdot 5^{c}\cdot 7^{d}-1\right|<\frac{5}{2^{k/2}}.
\end{equation}

We next apply linear forms in logarithms to lower bound the left--hand side of inequality \eqref{deslambda1 for s=4} as we did in \ref{n<2^k/2}. After some algebra, we finally get that
\[
k<3.27\times 10^{15}\log k.
\]
So, Mathematica gives us $k< 1.289\times 10^{17}$ and replacing this upper bound for $k$ in the first inequality of this lemma, we get the upper bound for $n$. 
\end{proof}

In order to prove Theorem \ref{teo2}, we now distinguish the following two cases.


\subsection{The case of small $k$}

Here we treat the cases when $k\in[2,900]$.  After finding an upper bound on $n$ for each $k\in[2,900]$ by using Lemma \ref{lema-s=4}, the next step is to reduce it. To do this, we let 
\[
z_1:=a\log 2+b\log 3+c\log 5+d\log 7-(n-1)\log\alpha-\log f_k(\alpha).
\]
Therefore, (\ref{deslambda for s=4}) can be rewritten as 
\begin{equation}\label{desz1}
|e^{z_1}-1|<\frac{2}{\alpha^{n-1}}.
\end{equation}

Observe that $z_1\neq 0$.  If $z_1>0$, then $e^{z_1}-1>0$, so from (\ref{desz1}) we obtain
\[
0<z_1<\frac{2}{\alpha^{n-1}},
\]
where we used the fact that $x\leq e^x-1$ for all $x\in \mathbb{R}$. Next, we treat the case $z_1<0$. First of all, note that if $n\geq 4$, then one checks easily that $2/\alpha^{n-1}<1/2$ for all $k\geq 2$. Thus, from (\ref{desz1}), we get that $|e^{z_1}-1|<1/2$ and therefore $e^{|z_1|}<2$. Since $z_1<0$, we obtain
\[
0<|z_1|\leq e^{|z_1|}-1=e^{|z_1|}|e^{z_1}-1|<\frac{4}{\alpha^{n-1}}.
\]
In any case, we have that the inequality
\begin{equation}   \label{estimate|z1|}
|z_1|<\frac{4}{\alpha^{n-1}},
\end{equation}
holds for all $k\geq 2$ and $n\geq 4$.  Observe that $|z_1|$ is an expression of the form
\[
|x_1\log 2+x_2\log 3+x_3\log 5+x_4\log 7+x_5\log\alpha+x_6\log f_k(\alpha)|, 
\]
where $x_1:=a$, $x_2:=b$, $x_3:=c$, $x_4:=d$, $x_5:=-(n-1)$, $x_6:=-1$ are integers with $\max\{|x_i|:1\leq i \leq 6 \}\leq n-1<7.73\times 10^{20}\,k^7\,\log^3k$ (see Lemma \ref{lema-s=4}). 

For each $k\in[2,900]$, we used the LLL algorithm to compute a lower bound for the smallest nonzero number of the form $|z_1|$, with integer coefficients $x_i$ not exceeding $7.73\times 10^{20}\,k^7\,\log^3k$ in absolute values.  We followed the method described in \cite[Section 2.3.5]{Cohen}, which provides such bound using the approximation for the shortest vector in the corresponding lattice obtained by LLL algorithm. 

After finding a good approximation of $\alpha$ and a lower bound for the minimal value of the nonzero number of the form $|z_1|$, we use \eqref{estimate|z1|} to get a new upper bound for $n$ which is less than the previous one.  After these computations, we concluded that the possible solutions $(n,k,a,b,c,d)$ of the equation (\ref{eq2}) for which $k$ is in the range $[2,900]$, all have $n< 1100$.    

Finally, we used simple programs written in Mathematica to display $F_n^{(k)}$  for the range $2\leq k \leq 900$ and $k+2\leq n< 1100$,  and checked that the solutions of the equation (\ref{eq2}) in this range are those given by Theorem \ref{teo2}.  This completes the analysis in the case $k\in[2,900]$. 


\subsection{The case of large $k$}

We now treat the case when $k>900$. Here, we would like to reduce our absolute upper bound on $k$ (see Lemma \ref{lema-s=4}), which is too large, by using again the LLL algorithm. In order to do this, let
\[
z_2:=(a-n+2)\log 2+b\log 3+c\log 5+d\log 7.
\]
So, we can rewrite \eqref{deslambda1 for s=4} as follows
\begin{equation}\label{desz2}
|e^{z_2}-1|<\frac{5}{2^{k/2}}.
\end{equation}
We can easily see that $z_2<0$ using the fact that the inequality $F_n^{(k)}<2^{n-2}$ holds for all $n\geq k+2$. Furthermore, since $k>900$, we have that $5/2^{k/2}<1/2$, thus it follows from \eqref{desz2} that $|e^{z_2}-1|<1/2$ which implies $e^{|z_2|}<2$. Since $z_2<0$, we obtain that
\begin{equation} \label{cota|z2|}
0<|z_2|\leq e^{|z_2|}-1=e^{|z_2|}|e^{z_2}-1|<\frac{10}{2^{k/2}}.
\end{equation}
On the other hand, observe that $|z_2|$ is an expression of the form
\begin{equation} \label{estimate|z2|}
|x_1\log 2+x_2\log 3+x_3\log 5+x_4\log 7|, 
\end{equation}
where now $x_1:=a-n+2$, $x_2:=b$, $x_3:=c$, $x_4:=d$.  From the second part of Lemma \ref{lema-s=4}, we have 
\begin{equation} \label{cotax_i}
\max\{|x_i|:1\leq i \leq 4 \} < n<2.795\times 10^{145}.
\end{equation}

We now performed the LLL algorithm to find a lower bound on the smallest nonzero number of the form \eqref{estimate|z2|} whose coefficients $x_i$ are integers satisfying \eqref{cotax_i}.   We got that this lower bound is $>1.462\times 10^{-439}$, which combined with \eqref{cota|z2|} gives that $k\leq 2922$.  This and the first part of Lemma \ref{lema-s=4} tell us that $n<8\times 10^{47}$.  With this new upper bound for $n$ we repeated the process; i.e., we use LLL algorithm once again to get a lower bound of $|z_2|$, where now the coefficients $x_i$ are integers satisfying
\[
\max\{|x_i|:1\leq i \leq 4 \} < n<8\times 10^{47}.
\]
Here we obtain $k\leq 980$.  After repeating the process 2 times more, we finally find that $k\leq 900$, which is a contradiction.

Thus, Theorem \ref{teo2} is proved.


\end{document}